\documentclass[a4paper, 11pt]{article}

\usepackage {amsmath,amsthm}
\usepackage{amsfonts}
\usepackage{amssymb}
\usepackage{bbm}
\usepackage{yfonts}
\usepackage{mathtools}
\usepackage[longnamesfirst]{natbib}

\newcommand{\x}{\times}
\newcommand{\p}{\partial}

\renewcommand{\d}{\mathrm{d}}
\newcommand{\e}{\mathrm{e}}
\newcommand{\irm}{\mathrm{i}}

\newcommand{\Lap}{\mathrm{\Delta}}
\newcommand{\Hd}{\dot H}
\newcommand{\eqnb}{\begin{equation}}
\newcommand{\eqnbs}{\begin{equation*}}
\newcommand{\eqnbsa}{\begin{equation*}\begin{aligned}}
\newcommand{\eqnba}{\begin{equation}\begin{aligned}}
\newcommand{\eqnbl}[1]{\begin{equation}\label{#1}}
\newcommand{\eqnbal}[1]{\begin{equation}\label{#1}\begin{aligned}}
\newcommand{\eqnes}{\end{equation*}}
\newcommand{\eqne}{\end{equation}}
\newcommand{\eqnesa}{\end{aligned}\end{equation*}}
\newcommand{\eqnea}{\end{aligned}\end{equation}}

\newcommand{\re}[1]{(\ref{#1})}
\newcommand{\oneChar}{\hspace{11pt}}
\newcommand{\comment}[1]{}

\newcommand{\RR}{\mathbb{R}}
\newcommand{\TT}{\mathbb{T}}

\newcommand{\ZZ}{\mathbb{Z}}

\newcommand{\NN}{\mathbb{N}}

\newtheorem{theorem}{Theorem}
\newtheorem{lemma}{Lemma}
\newtheorem{corollary}{Corollary}

\begin{document}
\title{Well-posedness for the diffusive 3D Burgers equations with initial data
in $H^{1/2}$}
\author{
Benjamin C. Pooley \footnote{
	BCP is supported by an EPSRC Doctoral Training Award.} \footnote {B.C.Pooley@warwick.ac.uk} \footnote{Mathematics Institute, University of Warwick, Coventry, CV4 7AL, UK} \and James C. Robinson  \footnotemark[3] \footnote{J.C.Robinson@warwick.ac.uk}}
\maketitle

\begin{abstract}
In this note we discuss the diffusive, vector-valued Burgers equations in a three-dimensional domain with periodic boundary conditions. We prove that given initial data in $H^{1/2}$ these equations admit a unique global solution that becomes classical immediately after the initial time. To prove local existence, we follow as closely as possible an argument giving local existence for the Navier--Stokes equations. The existence of global classical solutions is then a consequence of the maximum principle for the Burgers equations due to \cite{Kiselev_Ladyzhenskaya}.  

In several places we encounter difficulties that are not present in the corresponding analysis of the Navier--Stokes equations. These are essentially due to the absence of any of the cancellations afforded by incompressibility, and the lack of conservation of mass. Indeed, standard means of obtaining estimates in $L^2$ fail and we are forced to start with more regular data. Furthermore, we must control the total momentum and carefully check how it impacts on various standard estimates.
\end{abstract}

\section{Introduction}
We consider the three-dimensional, vector-valued diffusive Burgers equations. The equations, for a fixed viscosity $\nu>0$ and initial data $u_0$, are
\eqnbl{eqBurgers}
u_t+(u\cdot\nabla)u-\nu\Lap u =0,
\eqne
\eqnbl{eqData}
u(0)=u_0.
\eqne
 Working on the torus $\TT^3 = \RR^3/2\pi\ZZ^3$, we will investigate the existence and uniqueness of solutions $u$ to \re{eqBurgers}. Using the rescaling $\widetilde u(x,t)\coloneqq \nu u(x,\nu t)$, it suffices to prove well-posedness in the case $\nu=1$.

This system is well known and is often considered to be ``well understood''. However we have not found a self-contained account of its well--posedness in the literature, although for very regular data (with two H\"older continuous derivatives, for example) existence and uniqueness can be deduced from standard results about quasi--linear systems. We are particularly interested in an analysis parallel to the familiar treatment of the Navier--Stokes equations, which motivates our choice of function spaces here. 

It is interesting to note that we find some essential difficulties in treating this system which do not occur when  incompressibility is enforced,  i.e.\ for the Navier--Stokes equations. These prevent us from making the usual estimates that would give existence of ($L^2$-valued) weak solutions. We also find that taking initial data with zero average is not sufficient to ensure that the solution has zero average for positive times. This necessitates estimating the momentum and checking carefully that the methods applicable to  the Navier--Stokes equations have a suitable analogue.   

We begin with some brief comments on several relevant methods from the literature to motivate our discussion here. 

A maximum principle for solutions of the Burgers equations was proved by \cite{Kiselev_Ladyzhenskaya}. A simplified version of this result with zero forcing plays a key role in our argument, so we reproduce the proof here.
\begin{lemma}\label{lemMP}
If $u$ is a classical solution of the Burgers equations \re{eqBurgers} on a time interval $[a,b]$ then
\eqnbl{lemMP:bound1}
\sup_{t\in[a,b]}\|u(t)\|_{L^\infty}\leq \|u(a)\|_{L^\infty}.
\eqne
\end{lemma}
\begin{proof}
Fix $\alpha>0$ and let $v(t,x)\coloneqq \e^{-\alpha t}u(x,t)$ for all $x\in\TT^3$. Then $|v|^2$ satisfies the equation
\eqnbl{lemMP:vEqn}
\frac{\p }{\p t} |v|^2 + 2\alpha |v|^2 +u\cdot\nabla |v|^2 - 2v\cdot\Lap v=0 .
\eqne
Since $2v\cdot\Lap v = \Lap |v|^2 - 2|\nabla v|^2$ we see that if $|v|^2$ has a local maximum at $(x,t)\in(a,b]\x \TT^3$ then the left-hand side of \re{lemMP:vEqn} is positive unless $|v(x,t)|=0$. Hence
\[ 
\|u(t)\|_{L^\infty}\leq \e^{\alpha t}\|u(a)\|_{L^\infty}.
\]
Now \re{lemMP:bound1} follows because $\alpha>0$ was arbitrary.
\end{proof} 
In the discussion of well-posedness for \re{eqBurgers} in \cite{Kiselev_Ladyzhenskaya} the maximum principle is used with approximations obtained by considering discrete times and replacing the time derivatives by difference quotients. Unfortunately one of the steps there is incorrect. In the MathSciNet review, R. Finn relates a comment by L. Nirenberg that there is a flaw in the compactness argument given on p.\ 675. This error appears to be fatal.

Another well known approach comes by analogy with the Burgers equations in one dimension, namely the Cole--Hopf transformation, which gives analytic solutions by reducing the problem to solving a heat equation. Unfortunately this can only give gradient solutions, and since we wish to draw comparisons with the classical equations of fluid mechanics this is a significant drawback.  

There is a theorem in the book of \cite{LSU} (Chapter VII, Theorem 7.1) giving local well--posedness for a certain class of quasi-linear parabolic problems that includes \re{eqBurgers}. In that theorem the data and solutions are taken to have spatial H\"older regularity\footnote{The spaces in which solutions are found are actually defined by the existence and H\"older continuity (with certain exponent) of the mixed derivatives $D_t^rD_x^s$ for $2r+s < 2+\alpha$, where $D_x^s$ is any spatial derivative of order $s$.} at least $C^{2,\alpha}$ for some $\alpha\in(0,1)$. It is likely that a consequence is global well--posedness in these spaces, but this is not stated. A brief sketch of the proof is given, but it is quite different from any familiar method used for the Navier--Stokes equations. Moreover  and there is also no discussion of solutions gaining regularity that we will demonstrate (see Lemma \ref{lemBootstrap}).

To simplify several of the estimates proved later, we define for $s\geq 0$ the operator $\Lambda^s$ acting on $H^s(\TT^3)$ as follows. Let $f\in H^s(\TT^3)$ have the Fourier series
\[
f(x)=\sum_{k\in\ZZ^3}\hat f_k\e^{\irm k\cdot x}\in H^{s}(\TT^3),
\]
then we define
\[
\Lambda^s f(x) \coloneqq \sum_{k\in \ZZ^3}|k|^{s}\hat f_k\e^{\irm k\cdot x}\in L^2(\TT^3).
\]
Moreover we denote by $\|\cdot\|_s$ the seminorm $\|\Lambda^{s}\cdot\|_{L^2}$. This is of course compatible with the definition of the Sobolev norm; $\|\cdot\|_{H^s}$ is equivalent to $\|\cdot\|_{L^2}+\|\cdot\|_{s}$. Note that in the Fourier setting it is more usual to define an equivalent norm on $H^s$ by
\[
\|f\|=\left(\sum_{k\in\ZZ^3}(1+|k|^{2s})|\hat f_k|^2\right)^{1/2},
\]
but here we will usually consider $\|f\|_{L^2}$ and $\|f\|_s$ separately, when estimating $\|f\|_{H^s}$. We will also make use of the fact that $\|f\|_s\leq\|f\|_t$ if $0<s\leq t$ and that $\Lambda^2=(-\Lap)$,

We call $u\in C^0([0,T];H^{1/2})\cap L^2(0,T;H^{3/2})$ with $u_t\in L^2(0,T;H^{-1/2})$ a \textit{strong solution} of \re{eqBurgers} if, for any $\phi\in C^\infty(\TT^3)$
\eqnbl{eqWeakBurgers}
\langle u_t,\phi\rangle+((u\cdot\nabla) u,\phi)_{L^2} + (\nabla u,\nabla \phi)_{L^2}=0
\eqne
for almost all $t\in[0,T]$. Here $\langle\cdot,\cdot\rangle$ denotes the duality pairing of $H^{-1/2}(\TT^3)$ with $H^{1/2}(\TT^3)$. We consider the attainment of the initial data $u_0\in H^{1/2}$ in the sense of continuity into $H^{1/2}$.

We have chosen to use the term \textit{strong solution} here, even though (in the classical treatment of the Navier--Stokes equations) this usually refers to solutions in $L^\infty(0,T;H^1)\cap L^2(0,T; H^2)$. Indeed, we shall see that the solutions we find become classical, in a similar way to local strong solutions of the Navier--Stokes equations (see \cite{JCR_Rod_Sad_book}).

The reason for considering well-posedness in $H^{1/2}$ is that, as for the Navier--Stokes equations, if $u$ is a solution to the Burgers equations on $\RR^3$ then, for $\lambda>0$, so is
\[
u_\lambda\coloneqq\lambda u(\lambda^2 t, \lambda x)
\] 
and in three dimensions the seminorm $\|\cdot\|_{1/2}$ is invariant under this scaling. Therefore we would ideally consider solutions in the homogeneous space $\Hd^{1/2}$; however, as we will see, the zero-average property is not necessarily preserved in the solution. Fortunately we will also see that it is natural to control the ``creation of momentum'' by $\int_0^t\|u(s)\|_{1/2}\,\d s$ and we will check carefully that the relevant techniques from the analysis of the Navier--Stokes equations in $\Hd^{1/2}$ can be adapted. 

We will prove the following theorem.
\begin{theorem}\label{thmBurgersExistUniq}
Given $u_0\in H^{1/2}$, there exists a unique global strong solution $u\in C^0([0,\infty);H^{1/2})\cap L^2(0,\infty;H^{3/2})$.  Moreover, except at the initial time, $u\in C^1((0,\infty);C^0)\cap C^0((0,\infty);C^2)$ and is a classical solution.
\end{theorem}

We will prove this using Galerkin approximations to find unique local strong solutions and then, by bootstrapping, prove that the solution has enough regularity to rule out a blowup and deduce global existence using Lemma \ref{lemMP}.   

The well known arguments giving global existence of weak solutions to the Navier--Stokes equations (i.e.\ solutions with initial data in $L^2$ and regularity ${L^\infty(0,T;L^2)\cap L^2(0,T;H^1)}$) rely on the anti-symmetry 
\[((u\cdot\nabla)v,w)_{L^2}=-((u\cdot\nabla)w,v)_{L^2}
\] that is a consequence of incompressibility of $u$. This is not something we can make use of with the Burgers equations. We might instead try to find weak solutions using the maximum principle and the following estimate that holds for smooth solutions 
\eqnbl{eqLinftyIdea}
\frac{\d}{\d t}\|u\|_{L^2}^2 + \|\nabla u\|_{L^2}^2\leq \|u\|^2_{L^2}\|u\|_{L^\infty}^2.
\eqne
Making rigorous use of this would require a maximum principle to hold for the Galerkin approximations, but the proof of Lemma \ref{lemMP} does not work with a projection applied to the nonlinear term.

To avoid these difficulties we will start with more regular initial data (in $H^{1/2}$) and find classical solutions before making use of Lemma \ref{lemMP}. 

We separate the difficulties encountered in the proof of Theorem \ref{thmBurgersExistUniq} into two sections. In Section \ref{secH1} we prove global well-posedness for data $u_0\in H^1$. Here we use some standard \textit{a priori} estimates to find local strong solutions. We then bootstrap to show that the solution is classical after the initial time. This allows us to apply Lemma \ref{lemMP}, from which we derive better $H^1$ estimates that imply global existence.   

In Section \ref{secH1/2} we prove Theorem \ref{thmBurgersExistUniq} using techniques from \cite{Marin-Rubio_JCR_Sad_2013} to find a unique local solution for initial data $u_0\in H^{1/2}$. This solution instantly becomes classical, and hence global, by the results in Section \ref{secH1}.

\section{Solutions in $H^1$}\label{secH1}
We will use the method of Galerkin approximations. First we introduce some notation. For $n\in\NN$ let $P_n$ denote the projection onto the Fourier modes of order up to $n$, that is
\[
P_n\left(\sum_{k\in\ZZ^3}\hat u_k \e^{\irm x\cdot k}\right)=\sum_{|k|\leq n}\hat u_k \e^{\irm x\cdot k}.
\] Let $u_n=P_nu_n$ be the solution to
\eqnbl{eqGalerkinODEs}
\frac{\p u_n}{\p t} + P_n[(u_n\cdot\nabla)u_n] -\Lap u_n=0,
\eqne
with
\eqnb
u_n(0)=P_nu_0.
\eqne
For some maximal $T_n>0$ there exists a solution $u_n\in C^\infty([0,T_n)\x\TT^3)$ to this finite-dimensional locally-Lipschitz system of ODEs. 

As noted in the introduction, one of the interesting issues we encounter in this analysis of the Burgers equations is that we cannot guarantee that the solution has the zero-average property even if the initial data does. However we do have the following estimate to control the potential ``creation of momentum''. 
\begin{lemma}\label{lemMomentum}
Let $u,v$ be solutions of \re{eqGalerkinODEs}, with initial data $u_0$ and $v_0$ respectively. If $w=u-v$ and $w_0=u_0-v_0$ then
\eqnbal{lemMomentum:bound2}
\left|\int_{\TT^3} w(x,t)-w_0(x)\,\d x\right|\leq 8\pi^3\int_0^t\|w(s)\|_{1/2}(\|u(s)\|_{1/2}+\|v(s)\|_{1/2})\,\d s. 
\eqnea
In particular, taking $v\equiv 0$ yields
\eqnbl{lemMomentum:bound1}
\left|\int_{\TT^3} u(x,t)\,\d x\right|\leq8\pi^3 \int_0^t\|u(s)\|_{1/2}^2\,\d s+\left|\int_{\TT^3}u_0(x)\,\d x\right|. 
\eqne
\end{lemma}
\begin{proof}
For $k\in\ZZ^3$ denote the $k$th Fourier coefficients of $u$, $v$ and $w$ by $\hat u_k$, $\hat v_k$ and $\hat w_k$  respectively.
Considering the form of the equations satisfied by $u$ and $v$, we have
\eqnbsa
\frac{\d}{\d t}\int_{\TT^3} w(x,t)\,\d x &= - \int_{\TT^3} (u\cdot\nabla)w +(w\cdot\nabla)v\, \d x\\
&=-8\pi^3\irm\sum_{k\in\ZZ^3}\left(\overline{\hat u_k(t)}\cdot k\right)\hat w_k(t)+\left(\overline{\hat w_k(t)}\cdot k\right)\hat v_k(t),
\eqnesa
Hence
\eqnbsa
\,&\left|\frac{\d}{\d t}\int_{\TT^3} w(x,t)\,\d x\right|\leq 8\pi^3\sum_{k\in\ZZ^3}|\hat w_k||k|(|\hat u_k|+|\hat v_k|)\\
&\oneChar\leq8\pi^3\|w(t)\|_{1/2}(\|u(t)\|_{1/2}+\|v(t)\|_{1/2}),
\eqnesa
then \re{lemMomentum:bound2} follows after integrating with respect to $t$. 
\end{proof}

We will use this lemma to control the failure of equivalence of the norm $\|\cdot\|_{H^s}$ and the seminorm $\|\cdot\|_{s}$ for solutions of  \re{eqGalerkinODEs} as follows:
\eqnbl{eqEquivFail}
\|u_n(t)\|_{s}\leq \|u_n(t)\|_{H^s}\leq c\|u_n(t)\|_{s} + c\int_0^t\|u_n\|_{1/2}^2\,\d s+c\|u_0\|_{L^1},
\eqne
for some $c>0$ depending only on $s$. Here we have used the fact that $\int_{\TT^3}P_n u_0=\int_{\TT^3}u_0$. Note that we will occasionally use the equivalence of the seminorms $\|\cdot\|_s$ and $\|\cdot\|_{\dot H^s}$ when applicable. In particular for estimating the derivatives of sufficiently regular functions e.g.\ $\|\nabla u_n\|_{L^6}\leq c\|u_n\|_{2}$. 

We will prove the following special case of Theorem \ref{thmBurgersExistUniq}. The proofs of some estimates will only be sketched, if they are standard or when similar arguments are made in detail in Section \ref{secH1/2}. 
 
\begin{theorem}\label{thmBurgersExistUniqH1}
Given $u_0\in H^{1}$, there exists a unique global strong solution $u\in C^0([0,\infty);H^{1})\cap L^2(0,\infty;H^{2})$. Moreover, except at the initial time, $u\in C^1((0,\infty);C^0)\cap C^0((0,\infty);C^2)$ is a classical solution.
\end{theorem}

We first need a lower bound on the existence times for the Galerkin systems \re{eqGalerkinODEs} that is uniform, i.e.\ independent of $n$. For this we integrate the $L^2$ inner product of \re{eqGalerkinODEs} with $\Lambda^2 u_n$. Using the inequalities of H\"older and Young to control the nonlinear term, as we would for the Navier--Stokes equations, we obtain
\eqnbsa
\|u_n(t)\|^2_{1} + \int_0^t\|u_n(s)\|_{2}^2\,\d s\leq \|u_n(0)\|^2_{1} + c\int_0^t\|u_n(s)\|_{L^6}^4\|u_n(s)\|_{1}^2\,\d s
\eqnesa
for some $c>0$. Now by the embedding $H^1\hookrightarrow L^6$ and Lemma \ref{lemMomentum} 
\eqnbal{eqH1est1}
\,&\|u_n(t)\|^2_{1} + \int_0^t\|u_n(s)\|_{2}^2\,\d s\leq \|u_n(0)\|^2_{1} + c\int_0^t\|u_n(s)\|_{1}^6 \,\d s\\
&\oneChar\oneChar+ c\left(\int_0^t\|u_n(s)\|^2_{1/2}\,\d s + \|u_0\|_{L^1}\right)^4\int_0^t\|u_n(s)\|_{1}^2\,\d s\\
&\oneChar\leq \|u_n(0)\|^2_{1} + c\int_0^t\|u_n(s)\|_{1}^6 +c\int_0^t t^4\|u_n(s)\|^{10}_{1} + \|u_0\|_{L^1}^4\|u_n(s)\|_{1}^2\,\d s,
\eqnea
for some $c>0$. The last step made use of the fact that $\|u_n\|_{1/2}\leq \|u_n\|_{1}$ and the H\"older inequality
\[
\left(\int_0^t f(s)\,\d s\right)^5\leq t^{4}\int_0^t|f(s)|^5\,\d s.
\] 

Let us now impose the upper bound $t\leq 1$. Applying Young's inequality to the $\|u_n\|_{1}^6$ and $\|u_n\|_{1}^2$ terms of the last line of \re{eqH1est1} gives
\eqnbsa
\|u_n(t)\|^2_{1}&\leq \|u_n(0)\|^2_{1} + c\int_0^t\alpha^5\|u_n(s)\|^{10}_{1}+ \beta^5\,\d s,
\eqnesa
where $\alpha \coloneqq  (4+\|u_0\|_{L^1}^4)^{1/5}$ and $\beta\coloneqq (2+4\|u_0\|_{L^1}^4)^{1/5}$.
This gives an integral inequality of the form
\[
f(t)\leq f(0) + \int_0^t (af(s)+b)^5\,\d s.
\]
By solving this inequality we obtain 
\eqnbal{eqH1est2}
\|u_n(t)\|^2_{1}\leq\frac{\alpha \|u_0\|_{1}^2 + \beta}{\alpha\left(1-4\alpha t(\alpha\|u_0\|^2_{1}+\beta)^4\right)^{1/4}} - \frac{\beta}{\alpha}.
\eqnea
 Using Lemma \ref{lemMomentum} this estimate rules out a blowup of $u_n$ in $H^{1}$ before the time 
\eqnbl{eqTast}
T^\ast \coloneqq \frac{1}{4\alpha (\alpha\|u_0\|^2_{1}+\beta)^4}.
\eqne
It follows that there exists $T>0$, we can for example take $T=T^\ast/2$, such that $T_n\geq T$ for all $N$. At this point one might optimize the existence time $T$ by choosing a different upper bound where we assumed $t\leq 1$, but this is not necessary in what follows.

From  \re{eqH1est1} and \re{eqH1est2} we now have uniform bounds on $u_n\in L^\infty(0,T;H^1)$ and on $u_n\in L^2(0,T;H^2)$. By \re{eqGalerkinODEs} and a standard argument, we  also obtain a uniform bound on the time derivative $\p_t u_n\in L^2(0,T; L^2)$. Therefore, by the Aubin-Lions theorem, we may assume (after passing to a subsequence) that there exists $u\in C^0([0,T];H^1)$ such that $u_n\to u$ in $L^p(0,T;L^2)$ for all $p\in(1,\infty)$. A standard method shows that the limit $u$ is a strong solution. For the details of this type of argument, see, for example, \cite{ConstFoias}, \cite{Evans_PDEBook}, \cite{GaldiNotes} or \cite{JCR_RedBook}. 

 It can also be shown that this solution $u$ is a unique. We omit the proof of this here but we will see, in Section \ref{secH1/2}, that uniqueness requires even less regularity. Subject to this omission we have so far proved the following.
\begin{lemma}\label{lemLocalH1Solns}
Given $u_0\in H^1$ there exists $T >0$ (given by \re{eqTast}) such that the Burgers equations admit a unique strong solution $u$ on $[0,T]$ with initial data $u_0$. Moreover $u\in C^0([0,T];H^1)\cap L^2(0,T;H^2)$.  
\end{lemma} 
It follows that if $u\in L^\infty(0,T;H^{1/2})\cap L^2(0,T;H^{3/2})$ is a strong solution of the Burgers equations then $u\in C^0([\varepsilon,T];H^1)\cap L^2(\varepsilon,T;H^2)$ for any $\varepsilon\in(0,T)$. Therefore the following corollary is an easy consequence of \re{eqTast}.
\begin{corollary}\label{corBlowup}
If $u\in L^\infty(0,T;H^{1/2})\cap L^2(0,T;H^{3/2})$ is a local strong solution of \re{eqWeakBurgers} such that $T\in(0,\infty)$ is the maximum existence time (i.e.\ no strong solution exists on $[0,T+\varepsilon]$ for any $\varepsilon>0$), then $\mathrm{ess\ sup}_{(0,T)}\|u(t)\|_{H^1}=\infty$.
\end{corollary}

In order to prove Theorem \ref{thmBurgersExistUniqH1}, it remains to show that $u$ is, in fact, a global classical solution after the initial time. We will use a bootstrapping argument to obtain local classical solutions, followed by the maximum principle that will allow us to apply Corollary \ref{corBlowup} to show that the solution can be extended for an arbitrary length of time. 

The bootstrapping is carried out with the following lemma which is actually stronger than we will need. We will omit the proof as it is essentially the same as standard results about strong solutions of the Navier--Stokes equations that can be found in \cite{ConstFoias} and \cite{JCR_2006}, for example.

\begin{lemma}\label{lemBootstrap}
If the Galerkin approximations $u_n$ are uniformly bounded in $L^2(\varepsilon,T;H^{s+1})\cap L^\infty(\varepsilon,T;H^s)$ for $s>1/2$ and some $\varepsilon\geq 0$, then they are also bounded uniformly in $L^2(\varepsilon',T;H^{s+2})\cap L^\infty(\varepsilon',T;H^{s+1})$ for any $\varepsilon'\in(\varepsilon,T)$.
\end{lemma}

The uniform bounds on $u_n$ we have proved are sufficient to apply this lemma. In particular by applying it five times we have that for any $\varepsilon\in(0,T)$,  $(u_n)_{n=1}^\infty$ is a uniformly bounded sequence in $L^\infty(\varepsilon,T;H^6)$. In this case we have the following estimates on the time derivatives of $u_n$:
\[
\sup_{t\in(\varepsilon,T)}\left\|\frac{\p u_n}{\p t}\right\|_{H^4} \leq \sup_{t\in(\varepsilon,T)}\left(\|u_n(t)\|_{H^4}\|u_n(t)\|_{H^5}+ \|u_n(t)\|_{H^6}\right),
\]
and
\eqnbsa
\,&\sup_{t\in(\varepsilon,T)}\left\|\frac{\p^2u_n}{\p t^2}(t)\right\|_{H^2}\leq\sup_{t\in(\varepsilon,T)}\left(\left\|\frac{\p u_n}{\p t}(t)\right\|_{H^2}\|u_n(t)\|_{H^3}\right)\\ 
&\oneChar+\sup_{t\in(\varepsilon,T)}\left(\|u_n(t)\|_{H^2}\left\|\frac{\p u_n}{\p t}(t)\right\|_{H^3}+\left\|\frac{\p u_n}{\p t}(t)\right\|_{H^4}\right)
\eqnesa
It follows that $(u_n)_{n=1}^\infty$ is a bounded sequence in $H^1(\varepsilon,T; H^4)\cap H^2(\varepsilon,T;H^2)$. This regularity passes to the limit i.e.\ $u\in H^1(\varepsilon,T; H^4)\cap H^2(\varepsilon,T;H^2)$ and hence $u\in C^0([\varepsilon,T];C^2)\cap C^1([\varepsilon,T];C^0)$. This is enough regularity to conclude that $u$ is a local classical solution of the Burgers equations. Note that time regularity on these closed intervals follows by considering larger open intervals. 

To show that $u$ can be extended to a global solution we now use the maximum principle from Lemma \ref{lemMP}. Taking $\varepsilon>0$ as the initial time of the classical solution, as above, we have the following estimate for $t\in[\varepsilon,T]$:
\[
\frac{\d}{\d t}\|u\|_{1}^2 \leq 2|((u\cdot\nabla)u,\Lambda^2 u)_{L^2}| - 2\|u\|_{2}^2 \leq \|u\|_{L^\infty}^2\|u\|_{1}^2. 
\]
Therefore 
\[
\sup_{t\in[\varepsilon,T]}\|u(t)\|_{1} \leq \|u(\varepsilon)\|_{1}\e^{t\|u(\varepsilon)\|_{L^\infty}^2/2}.
\]
This rules out the blowup of $u$ in the $H^1$ norm as $t\to T$, hence by Corollary \ref{corBlowup} the solution can be extended over $[\varepsilon,\infty)$, as required.

This completes the proof of Theorem \ref{thmBurgersExistUniqH1}, subject to a proof of uniqueness which can be found in the next section. 

\section{Proof of Theorem \ref{thmBurgersExistUniq}}\label{secH1/2}

We now set about proving Theorem \ref{thmBurgersExistUniq}. The argument will follow the same pattern as the previous section. That is, we will prove that for initial data in $H^{1/2}$ there exists $T$, independent of $n$, such that the Galerkin systems have solutions on an interval $[0,T]$. We then deduce the existence and uniqueness of a local strong solution $u\in  L^2(0,T;H^{3/2})\cap C^0([0,T];H^{1/2})$ of the Burgers equations. This is regular enough that global solutions can be obtained by appealing to the case of $H^1$ data. 

As in the previous section we denote by $u_n\in C^\infty([0,T_n)\x \TT^3)$ the unique solution to the Galerkin system \re{eqGalerkinODEs} with maximal existence time $T_n$. We allow the case $T_n=\infty$ but note that if $T_n<\infty$ then we necessarily have $\|u_n(t)\|_{L^2}\to\infty$ as $t\nearrow T_n$. 

Following \cite{Marin-Rubio_JCR_Sad_2013} (see also \cite{Chemin_book_2006}, \cite{Calderon_1990} and \cite{Fabes_etal_1972}) we split \re{eqGalerkinODEs} into a heat part, and a nonlinear part with zero initial data. Let $v$ be the periodic solution of the heat equation with initial data $u_0$, then $v_n\coloneqq P_nv$ satisfies
\eqnbs
\frac{\p}{\p t} v_n + \Lap v_n =0,\; v_n(0)=P_nu_0.
\eqnes 
Let $w_n\coloneqq u_n-v_n$, then $w_n$ satisfies 
\eqnbl{eqW}
\frac{\p}{\p t}w_n+P_n[(u_n\cdot\nabla)u_n] - \Lap w_n =0,\;w_n(0)=0.
\eqne
For $v_n$ and $t\in[0,T_n)$ we have the estimate
\eqnbal{eqVBound}
\sup_{s\in[0,t]}\|v_n(s)\|_{H^{1/2}}^2+2\int_0^t\|v_n(s)\|_{H^{3/2}}^2\,\d s\leq \|P_nu_0\|_{H^{1/2}}^2.
\eqnea
Integrating \re{eqW} against $\Lambda^{1} w_n$, gives
\eqnbal{eqWCalc1}
\,&\|w_n(t)\|_{1/2}^2 + 2\int_0^t\|w_n(s)\|_{3/2}^2\,\d s\\
&\leq \int_0^t \|u_n(s)\|_{L^6}\|\nabla u_n(s)\|_{L^2}\|\Lambda^{1}w_n(s)\|_{L^3}\,\d s\\
&\leq c_1\int_0^t\|u_n(s)\|_{H^1}\|u_n(s)\|_{1}\|w_n(s)\|_{3/2}\,\d s\eqqcolon I_0.
\eqnea
For some $c_1>0$. Now by Lemma \ref{lemMomentum} and the definition of $w_n$, 
\eqnbsa
\ &\|u_n(t)\|_{H^1}\|u_n(t)\|_{1}\leq c_2\left(\|u_n(t)\|_{1}+\int_0^t\|u_n(s)\|_{1/2}^2\,\d s+\|u_0\|_{L^1}\right)\|u_n(t)\|_{1}\\
&\oneChar\leq 2c_2(\|v_n(t)\|_{1}^2+\|w_n(t)\|^2_{1}) \\
&\oneChar\oneChar+ c_2(\|v_n(t)\|_{1}+\|w_n(t)\|_{1})\left(\int_0^t\|u_n(s)\|_{1/2}^2\,\d s+\|u_0\|_{L^1}\right)\eqqcolon I_1 + I_2
\eqnesa
for some $c_2>0$.
To estimate $I_1\x\|w_n\|_{3/2}$ we apply Young's inequality,
\eqnbsa
\,&\|w_n(t)\|_{3/2}(\|v_n(t)\|_{1}^2+\|w_n(t)\|^2_{1})\\
&\oneChar\leq\frac{1}{4c_1c_2}\|w_n(t)\|_{3/2}^2+ c \|v_n(t)\|_{1}^4 +  \|w_n(t)\|_{3/2}\|w_n(t)\|^2_{1},
\eqnesa
for some $c>0$. Also by several applications of Young's inequality, we estimate $I_2\x\|w_n\|_{3/2}$ as follows:
\eqnbsa
\,& \|w_n(t)\|_{3/2}(\|v_n(t)\|_{1}+\|w_n(t)\|_{1})\left(\int_0^t\|u_n(s)\|_{1/2}^2\,\d s+\|u_0\|_{L^1}\right)\\
&\leq\frac{1}{2}\|w_n(t)\|_{3/2}\left(\|v_n(t)\|_{1}^2 +\|w_n(t)\|_{1}^2\right)\\
&\oneChar+ \|w_n(t)\|_{3/2}\left(\int_0^t\|u_n(s)\|_{1/2}^2\,\d s+\|u_0\|_{L^1}\right)^2\\
&\leq\frac{1}{2c_1c_2}\|w_n(t)\|_{3/2}^2 + c\|v_n(t)\|^4_{1} +\frac{1}{2}\|w_n(t)\|_{3/2}\|w_n(t)\|_{1}^2\\
&\oneChar+c\left(\int_0^t\|u_n(s)\|_{1/2}^2\,\d s+\|u_0\|_{L^1}\right)^4
\eqnesa 
for some $c>0$.  To control the $\|w_n\|_{3/2}\|w_n\|^2_{1}$ terms in the last two estimates we use the interpolation 
\eqnbsa
 \,&\int_0^t\|w_n(s)\|_{3/2}\|w_n(s)\|^2_{1}\,\d s\leq \int_0^t\|w_n(s)\|_{3/2}^2\|w_n(s)\|_{1/2}\,\d s\\
&\oneChar\leq \frac{1}{5c_1c_2}\sup_{s\in[0,t]}\|w_n(s)\|_{1/2}^2 + c\left(\int_0^t\|w_n(s)\|^2_{3/2}\,\d s\right)^2
\eqnesa
for some $c>0$. Recombining these estimates of  $I_0$ and multiplying by $2$, \re{eqWCalc1} becomes
\eqnbal{eqWBound}
\,&\sup_{s\in[0,t]}\|w_n(s)\|_{1/2}^2+2\int_0^t\|w_n(s)\|^2_{3/2}\,\d s\\
&\leq a_1\int_0^t\|v_n(s)\|_{1}^4\,\d s + a_2\left(\int_0^t\|w_n(s)\|_{3/2}^2\,\d s\right)^2 \\
&\oneChar+ a_3\int_0^t\left(\int_0^s\|u_n(r)\|^2_{1/2}\,\d r+\|u_0\|_{L^1}\right)^4\,\d s ,
\eqnea
where $a_1,a_2,a_3>0$ are independent of $n$ and $t$. 
To simplify the last term we fix $c'>0$ such that
\eqnbsa
\int_0^t\left(\int_0^s\|u_n(r)\|^2_{1/2}\,\d r\right)^4\,\d s \leq c' t\left(\int_0^t\|v_n(s)\|^2_{1/2}\,\d s\right)^4 + c' t^5\sup_{s\in[0,t]}\|w_n\|^8_{1/2}.
\eqnesa
Thus \re{eqWBound} becomes
\eqnbal{eqWBound2}
\,&\sup_{s\in[0,t]}\|w_n(s)\|_{1/2}^2+2\int_0^t\|w_n(s)\|^2_{3/2}\,\d s\\
&\leq a_1\int_0^t\|v(s)\|_{1}^4\,\d s + a_2\left(\int_0^t\|w_n(s)\|_{3/2}^2\,\d s\right)^2 + a_3c' t\|u_0\|_{L^1}^4 \\
&\oneChar+ a_3c' t\left(\int_0^t\|v(s)\|^2_{1/2}\,\d s\right)^4 + a_3c' t^5\sup_{s\in[0,t]}\|w_n(s)\|^8_{1/2}.
\eqnea
This used the fact that $\|v_n(t)\|_{\sigma}$ is an increasing function of $n$ for all $\sigma\geq 0$ and $t\in[0,T^n]$.

We next use \re{eqWBound2} to find a uniform lower bound on the maximal existence time $T_n$, of $u_n$. It suffices to consider the case $T_n<\infty$. Comparing the $w_n$ terms on the left-hand and right-hand sides of \re{eqWBound2},  we define
\[
E(t)\coloneqq{a_2}\left(\int_0^t\|w_n(s)\|_{3/2}^2\,\d s\right) +a_3c' t^5\sup_{s\in[0,t]}\|w_n(s)\|^6_{1/2}
\]
and set
\[
\tau_n\coloneqq\sup\left\{t\in[0,T_n): E(t) \leq 1\right\}.
\]
Observe that $\tau_n<T_n$ since $E$ is continuous and $E(t)\to\infty$ as $t\to T_n$, because $\|w_n(t)\|_{L^2}$ must blow up as $t\to T_n$. This also means that $E(\tau_n)=1$.

As notation for the terms in the right-hand side of \re{eqWBound2} that do not depend on $w_n$, we define
\[
F(t)\coloneqq a_1\int_0^t\|v(s)\|_{1}^4\,\d s + a_3c' t\left(\int_0^t\|v(s)\|^2_{1/2}\,\d s\right)^4 + a_3c' t\|u_0\|_{L^1}^4.
\]
Note that $F(t)$ is a continuous increasing function that is positive except at $t=0$ (assuming that $u_0$ is non-zero). We now define  
\[
T \coloneqq \sup \left \{t\in[0,\infty):F(t)<\min\left(\frac{1}{(16a_3c't^5)^{1/3}},\frac{1}{2a_2}\right) \right\}.
\]
It is easy to see that $T>0$ and is independent of $n$. We will show that $T_n\geq T$ for all $n$.
Suppose, for contradiction, that $\tau_n<T$, then by \re{eqWBound2},
\[
\frac{1}{2}\sup_{s\in[0,\tau_n]}\|w_n(s)\|_{1/2}^2+\int_0^{\tau_n}\|w_n(s)\|^2_{3/2}\,\d s\leq F(\tau_n).
\]
Hence
\[
E(\tau_n)=a_2\left(\int_0^{\tau_n}\|w_n(s)\|_{3/2}^2\,\d s\right) +a_3c' {\tau_n}^5\sup_{s\in[0,\tau_n]}\|w_n(s)\|^6_{1/2}<1.
\]
This is a contradiction since we showed that $E(\tau_n)=1$.

We have shown that $T_n\geq T$ for all $n$. Furthermore, arguing as above we have
\[
\frac{1}{2}\sup_{s\in[0,T]}\|w_n(s)\|_{1/2}^2+\int_0^{T}\|w_n(s)\|^2_{3/2}\,\d s\leq F(T).
\]
Thus $(u_n)_{n=1}^\infty$ is uniformly bounded in $L^2(0,T;H^{3/2})$ and $L^\infty(0,T;H^{1/2})$; moreover this regularity implies that $\p_t u_n\in L^2(0,T;H^{-1/2})$, by a routine calculation.
Proceeding as before with a standard compactness argument one can show that $u$ is a local strong solution in the sense of \re{eqWeakBurgers}.

Next we prove that this local solution is unique (this argument also applies to give the uniqueness we claimed in Section \ref{secH1}). Suppose that $u$ and $v$ are strong solutions to \re{eqWeakBurgers} with the same initial data. Set $w=u-v$ then taking the product of the equation satisfied by $w$ with $2\Lambda^{1}w$ yields the estimate
\eqnbal{eqnUniqEst1}
\|w(t)\|_{1/2}^2 +2\int_0^t\|w(s)\|_{1/2}^2\,\d s &\leq c\int_0^t\|u(s)\|_{L^6}\|w(s)\|_{1}\|w(s)\|_{3/2}\,\d s\\
&\oneChar+ c\int_0^t \|w(s)\|_{H^{1/2}}\|v(s)\|_{3/2}\|w(s)\|_{3/2}\,\d s . 
\eqnea
For the first term we use interpolate $\|w\|_{1}^2\leq\|w\|_{1/2}\|w\|_{3/2}$ and Young's inequality to obtain:
\eqnbal{eqnUniqEst2}
c\|u(s)\|_{L^6}\|w(s)\|_{1}\|w(s)\|_{3/2}\leq c\|u(s)\|_{H^1}^4\|w(s)\|_{1/2}^2 + \|w(s)\|_{3/2}^2.
\eqnea
For the second we make use of Lemma \ref{lemMomentum} and the fact that $w(0)=0$:
\eqnbal{eqnUniqEst3}
\,&c\|w(s)\|_{H^{1/2}}\|v(s)\|_{3/2}\|w(s)\|_{3/2}\leq c\|v(s)\|_{3/2}^2\|w(s)\|_{1/2}^2 + \|w(s)\|_{3/2}^2\\
&\oneChar +c\|v(s)\|_{3/2}^2\left(\int_0^s\|w(r)\|_{1/2}\left(\|u(r)\|_{1/2}+\|v(r)\|_{1/2}\right)\,\d r\right)^2.
\eqnea
The integral over $[0,t]$ of the last term in \re{eqnUniqEst3} is at most
\[
c\left( \int_0^t\|v(s)\|_{3/2}^2\d s\right)\left(\int_0^t \|w(s)\|_{1/2}^2\d s\right)\left(2\int_0^t\|u(s)\|_{1/2}^2+\|v(s)\|_{1/2}^2\d s\right).
\] 
As $u\in L^4(0,T;H^{1/2})$ and $v\in L^2(0,T;H^{1/2})\cap L^2(0,T;H^{3/2})$, this together with \re{eqnUniqEst1}, \re{eqnUniqEst2} and  \re{eqnUniqEst3} imply that
\[
\|w(t)\|_{1/2}^2 \leq \int_0^t G(s)\|w(s)\|_{1/2}^2\,\d s
\]
for some $G\in L^1(0,T)$. A Gronwall inequality now implies that, since $\|w(0)\|_{1/2}=0$, $\|w(t)\|_{1/2}=0$ for all $t\in[0,T]$. Uniqueness now follows using Lemma \ref{lemMomentum}. 

We have proved the following.
\begin{lemma}\label{lemExistUniq}
For $u_0\in H^{1/2}$ there exists $T>0$ and a unique strong solution $u\in L^{2}(0,T;H^{3/2}) \cap C^0([0,T];H^{1/2})$ to the Burgers equations, in the sense of \re{eqWeakBurgers}. 
\end{lemma}

Fix a representative of $u$ that is continuous with respect to time into $H^{1/2}$. For almost every $t\in[0,T]$, we certainly have $u(t)\in H^1$, in which case  we can apply Theorem \ref{thmBurgersExistUniqH1} to obtain global classical solutions (on $(t,\infty)$) with initial data $u(t)$. By continuity of $u$ and uniqueness of local strong solutions these classical solutions agree with $u$ on their common domain. This completes the proof of Theorem \ref{thmBurgersExistUniq}.

\section{Conclusions}
We have shown that in the case of periodic boundary conditions the vector-valued diffusive Burgers equations have a unique solution given initial data in $H^{1/2}$. These solutions become classical immediately after the initial time and can be extended globally. 

The results here contrast with classical results about the Navier--Stokes equations, which have thus far only been shown to have \textit{local} well-posedness in $\Hd^{1/2}$ (see \cite{Marin-Rubio_JCR_Sad_2013} or \cite{Chemin_book_2006}). The main difference between these two systems seems to be the maximum principle for the Burgers equations. In other respects the analysis is slightly more straightforward in the case of Navier--Stokes, since we can make use of incompressibility. 

In several places we appealed to the analysis of Fourier series but otherwise we have not used the periodicity of the solution in an 
essential way. Therefore we might expect similar results to hold on $\RR^3$ or on other domains. 

As discussed in the introduction we have not been able to find weak solutions for less regular data $u_0\in L^2$ and it would be interesting to seek well-posedness of the Burgers equations in the various critical spaces that are often used to find local well-posedness results for the Navier--Stokes equations. Some examples of such spaces are: $L^3$ (\cite{KatoLp_1984}), certain Besov spaces (\cite{CannoneMeyerPlanchon_1994})  and $\mathrm{BMO}^{-1}$ (\cite{Koch_Tataru_2001}). 

 Irrespective of any approach in $L^3$ and the other aforementioned spaces, the existence of a maximum principle leads us to ask whether initial data $u_0\in L^\infty\cap L^2$ is enough to deduce local or global well-posedness. The main obstacle to doing this seems to be that we must find classical solutions before applying the maximum principle, since the maximum principle does not seem to pass to the Galerkin approximations. Using another system of approximations might avoid this difficulty, for example, a variation on the time-discretisation approach of \cite{Kiselev_Ladyzhenskaya}.

\bibliographystyle{authordate1}

\end{document}